\documentclass[12pt,reqno]{amsart}
\usepackage{amsmath}
\usepackage{amsthm}
\usepackage{amsfonts}
\usepackage{amssymb}
\usepackage{mathrsfs}
\usepackage[margin=1.35in, top =1.25in, bottom = 1.25in]{geometry}
\usepackage{setspace}
\usepackage{tikz}
\usetikzlibrary{matrix, arrows}
\usepackage{tikz-cd}
\usepackage{enumitem}
\usepackage[colorlinks, breaklinks]{hyperref}
\hypersetup{linkcolor=black,citecolor=blue,filecolor=black,urlcolor=black} 
\usepackage{pdflscape}
\usepackage{array}
\usepackage{adjustbox}

\newcommand{\abs}[1]{\left\vert #1 \right\vert}
\newcommand{\ceiling}[1]{\left\lceil #1 \right\rceil}
\DeclareMathOperator{\ch}{char}				
\DeclareMathOperator{\codim}{codim}
\newcommand{\dd}{\partial}					
\DeclareMathOperator{\Ext}{Ext}
\newcommand{\floor}[1]{\left\lfloor #1 \right\rfloor}
\DeclareMathOperator{\hgt}{ht}					
\DeclareMathOperator{\Hom}{Hom}
\DeclareMathOperator{\Id}{Id}					
\newcommand{\iso}{\cong}					
\newcommand{\kk}{k}						
\newcommand{\NN}{\mathbb{N}}
\DeclareMathOperator{\pd}{pd}					
\renewcommand{\phi}{\varphi}
\DeclareMathOperator{\poincare}{P}
\renewcommand{\P}{\poincare}					
\newcommand{\PP}{\mathbb{P}}
\DeclareMathOperator{\reg}{reg}				
\DeclareMathOperator{\Sym}{Sym}
\DeclareMathOperator{\Syz}{Syz}
\newcommand{\tensor}{\otimes}
\renewcommand{\tilde}[1]{\widetilde{#1}}
\DeclareMathOperator{\Tor}{Tor}
\DeclareMathOperator{\type}{type}
\newcommand{\ZZ}{\mathbb{Z}}

\newtheorem{thm}{Theorem}[section]
\newtheorem{prop}[thm]{Proposition}
\newtheorem{lemma}[thm]{Lemma}
\newtheorem{cor}[thm]{Corollary}

\theoremstyle{definition}
\newtheorem{rmk}[thm]{Remark}
\newtheorem{notation}[thm]{Notation}

\newtheorem{example}[thm]{Example}

\newtheorem{question}[thm]{Question}
\newtheorem*{ack}{Acknowledgements}

\numberwithin{equation}{section}
\numberwithin{figure}{section}

\begin{document}

\title{Quadratic Gorenstein Rings and the Koszul Property I}
\date{}

\author{Matthew Mastroeni}
\address{Department of Mathematics, Oklahoma State University, Stillwater, OK 74078}
\email{mmastro@okstate.edu}

\author{Hal Schenck}
\thanks{Schenck supported by NSF 1818646.}
\address{Department of Mathematics, Iowa State University, Ames, IA 50011}
\email{hschenck@iastate.edu}

\author{Mike Stillman}
\thanks{Stillman supported by NSF 1502294.}
\address{Department of Mathematics, Cornell University, Ithaca, NY 14850}
\email{mike@math.cornell.edu}

\subjclass[2000]{Primary 13D02; Secondary 14H45, 14H50} 
\keywords{Syzygy, Koszul algebra, Gorenstein algebra.}

\maketitle

\begin{abstract}
Let $R$ be a standard graded Gorenstein algebra over a field presented by quadrics. In \cite{Gröbner:flags:and:Gorenstein:algebras}, Conca-Rossi-Valla show that such a ring is Koszul if $\reg R \leq 2$ or if $\reg R = 3$ and $c=\codim R \leq 4$, and they ask whether this is true for $\reg R = 3$ in general. We determine sufficient conditions on a non-Koszul quadratic Cohen-Macaulay ring $R$ that guarantee the Nagata idealization $\tilde{R} = R \ltimes \omega_R(-a-1)$ is a non-Koszul quadratic Gorenstein ring. We use this to negatively answer the question of \cite{Gröbner:flags:and:Gorenstein:algebras}, constructing non-Koszul quadratic Gorenstein rings of regularity 3 for all $c \geq 9$.
\end{abstract}

\begin{spacing}{1.15}
\section{Introduction}

Let $I$ be a homogeneous ideal in a standard graded polynomial ring $S$ over a field $\kk$, and set $R = S/I$.  In this paper, we study the relationship between two conditions that impose extraordinary constraints on the homological properties of $R$, namely the Gorenstein and Koszul properties.  The ring $R$ is \emph{Gorenstein} if it is Cohen-Macaulay and its canonical module is isomorphic to a shift of $R$:
\[
\omega_R = \Ext_S^c(R,S)(-n) \iso R(a)
\]
where $\hgt I = c$ and $\dim S = n$. This implies that the \emph{graded Betti numbers} $\beta_{i,j}^S(R) = \dim_{\kk} \Tor_i^S(R,\kk)_j$
have a symmetry 
\begin{equation} \label{Gorenstein:Betti:table:symmetry}
\beta_{i,j}^S(R) = \beta^S_{c-i,a+n-j}(R)
\end{equation}
for all $i, j$.  

On the other hand, $R$ is \emph{Koszul} if the ground field $R/R_+ \iso \kk$ has a linear free resolution over $R$.  That is, we have $\beta^R_{i,j}(\kk) = \dim_{\kk} \Tor_i^R(\kk, \kk)_j = 0$ for all $i$ and $j$ with $j \neq i$.  Koszul algebras have strong duality properties such as a close relationship between the Hilbert series of $R$ and the Poincar\'e series of $\kk$ over $R$.  In particular, when $R$ is Koszul, its defining ideal $I$ must be generated by homogeneous forms of degree two, and there are significant restrictions \cite{Backelin} \cite[3.2]{free:resolutions:over:Koszul:algebras} \cite[3.4, 4.2]{Koszul:algebras:defined:by:3:quadrics} on the graded Betti numbers of $R$ over $S$ compared to general quadratic algebras.  Moreover, Koszul algebras appear as many rings of interest in commutative algebra, topology, and algebraic geometry; they include quotients by any quadratic monomial ideal, the coordinate rings of Grassmannians \cite{Kempf} and sufficiently small general sets of points in projective space \cite{points:in:projective:space}, and all suitably high Veronese subrings of any standard graded algebra \cite{high:Veronese:subrings:are:Koszul}.  We refer the interested reader to the surveys \cite{Fröberg:Koszul:algebras:survey} and \cite{Koszul:algebras:and:their:syzygies} and the references therein for further details about Koszul algebras.

A particularly important motivation for the present article is that, if $R$ is the homogeneous coordinate ring of a general curve $C$ of genus $g \geq 5$ in its canonical embedding, Vishik and Finkelberg prove that $R$ is Koszul in \cite{Vishik:Finkelberg}.  Building on this, Polishchuk shows that $R$ is Koszul if $C$ is not a plane quintic, hyperelliptic, or trigonal in \cite{Polishchuk}.  Such rings are also quadratic and Gorenstein by \cite[9.5]{geometry:of:syzygies}, so a natural question is: 

\begin{question}
Are quadratic Gorenstein rings always Koszul?
\end{question}  

Unfortunately, in \cite{Matsuda}, Matsuda shows that this is not the case by constructing a quadratic Gorenstein toric ring of regularity four and codimension seven which is not Koszul.  Nonetheless, there is some evidence that quadratic Gorenstein rings are Koszul, at least when $R$ is a complete intersection or has small Castelnuovo-Mumford regularity (see \S \ref{CM:rings:background} for a review of this concept): \\[-5 pt]
\begin{itemize}
\item Every quadratic complete intersection is Koszul.  This was first proved by Tate in \cite{Tate}; see \cite[1.19]{Koszul:algebras:and:their:syzygies} for an easier argument due to Caviglia. \\[-5 pt]

\item If $\reg R = 2$, then $R$ is Koszul by \cite[2.12]{Gröbner:flags:and:Gorenstein:algebras}. \\[-5 pt]

\item If $\reg R = 3$ and $\codim R = \hgt I \leq 5$, then $R$ is Koszul.  This follows from \cite[6.15]{Gröbner:flags:and:Gorenstein:algebras} and more recently by \cite{El:Khoury:Kustin} when $\codim R = 4$ and by \cite{Caviglia} when $\codim R = 5$.  \\[-5 pt]

\item If $\reg R=3$, and $\dim R = 2$, then $R$ is the canonical ring of a curve by \cite[9C.2]{geometry:of:syzygies} so that $R$ is Koszul by \cite{Vishik:Finkelberg} and \cite{Polishchuk}. \\[-5 pt]
\end{itemize}
Note that the symmetry \eqref{Gorenstein:Betti:table:symmetry} of the free resolution of a quadratic Gorenstein ring forces $\reg R \geq 2$ unless $R$ is a hypersurface and that $R$ is also Koszul in that case by the first bullet above.  These results led Conca, Rossi, and Valla to pose the following question.

\begin{question}[{\cite[6.10]{Gröbner:flags:and:Gorenstein:algebras}}] \label{quadratic:Gorenstein:rings:of:regularity:3:are:Koszul} 
If $R$ is a quadratic Gorenstein ring with $\reg R = 3$, is $R$ Koszul? 
\end{question}

More generally, one might ask:

\begin{question} \label{when:are:quadratic:Gorenstein:rings:Koszul}
For which positive integers $c$ and $r$ is every quadratic Gorenstein ring $R$ with $\codim R = c$ and $\reg R = r$ Koszul?
\end{question}

Matsuda's example in \cite{Matsuda} does not address the Conca-Rossi-Valla question since the toric ring he constructs has regularity four.  We give a negative answer (Example \ref{counterexample:to:Conca-Rossi-Valla}) to Question \ref{quadratic:Gorenstein:rings:of:regularity:3:are:Koszul} with codimension nine and a partial answer to Question \ref{when:are:quadratic:Gorenstein:rings:Koszul}.  In fact, our main result (Theorem \ref{non-Koszul:quadratic:superlevel:algebras:yield:non-Koszul:quadratic:Gorenstein:rings}) provides a machine for producing lots of examples of non-Koszul quadratic Gorenstein rings by deducing conditions on a quadratic Cohen-Macaulay ring such that the idealization $\tilde{R} = R \ltimes \omega_R(-a-1)$ is a non-Koszul quadratic Gorenstein ring.  

After introducing the necessary background on Cohen-Macaulay rings in \S \ref{CM:rings:background}, we prove our main result in \S \ref{main:result} and apply it in \S \ref{superlevel:examples} to give many examples of non-Koszul quadratic Gorenstein rings.  As a consequence, we prove the existence of non-Koszul quadratic Gorenstein rings $R$ with $\reg R = 3$ and $\codim R = c$ for all $c \geq 9$ in characteristic zero, which is the setting originally considered in \cite{Gröbner:flags:and:Gorenstein:algebras}.

\begin{notation}
Throughout the remainder of the paper, we use the following notation unless specifically stated otherwise.  Let $\kk$ be a fixed ground field of any characteristic, $S$ be a standard graded polynomial ring over $\kk$, and $I \subseteq S$ be a graded ideal such that $R = S/I$ is Cohen-Macaulay.  We set $\omega = \omega_R(-a-1)$, where $a = a(R)$ is the $a$-invariant of $R$, and $\tilde{R} = R \ltimes \omega$ denotes the idealization of $\omega$.  Recall that the ideal $I$ is called \emph{nondegenerate}\index{nondegenerate} if it does not contain any linear forms.  We can always reduce to a presentation for $R$ with $I$ nondegenerate by killing a basis for the linear forms contained in $I$, and we will assume that this is the case throughout.  We denote the irrelevant ideal of $R$ by $R_+ = \bigoplus_{n \geq 1} R_n$.
\end{notation}

\section{Background on Cohen-Macaulay rings} \label{CM:rings:background}

In this section, we briefly recall some invariants associated to standard graded algebras and discuss how they specifically relate to Cohen-Macaulay rings.  We refer the reader to \cite{Bruns:Herzog} and \cite{Brodmann:Sharp} for further details and any unexplained terminology.  

Let $R = S/I$ be a standard graded algebra of dimension $d$.  An important invariant of $R$ is its (Castelnuovo-Mumford) \emph{regularity}
\begin{align}
\label{regularity}
\begin{split} 
\reg R &= \max\{j \mid H^i_{R_+}(R)_{j-i} \neq 0  \; \text{for some}\; i \}
\\
&= \max\{j \mid \beta_{i,i+j}^S(R) \neq 0 \; \text{for some}\; i \} 
\end{split}
\end{align}
where $H^i_{R_+}(R)$ denotes the $i$-th local cohomology module of $R$ with respect to its irrelevant ideal.  Recall that the injective hull of $\kk$ over $R$ is the injective $R$-module 
\[E = E_R(\kk) = {}^*\Hom_\kk(R, \kk). 
\]
This is a graded $R$-module where ${}^*\Hom_\kk(R, \kk)_j \iso \Hom_\kk(R_{-j}, \kk)$ is the set of $\kk$-linear maps $R \to \kk$ of degree $j$.  The \emph{canonical module} of $R$ is the Matlis dual of its top local cohomology module
\[ \omega_R = {}^*\Hom_R(H^d_{R_+}(R), E) \iso {}^*\Hom_\kk(H^d_{R_+}(R), \kk) \]
A closely related quantity is the \emph{$a$-invariant} of $R$, which is defined by
\[ a(R) = \max\{j \mid H^d_{R_+}(R)_j \neq 0 \} = -\min\{j \mid (\omega_R)_j \neq 0 \} \]
so that $\omega_R$ is generated in degrees at least $-a(R)$.  As a consequence, we have an inequality 
\[a(R) + \dim R \leq \reg R.\]
When $R$ is Cohen-Macaulay, it is well known that $H^i_{R_+}(R) = 0$ for all $i \neq d$ so that the preceding becomes an equality.  Moreover, we say that a Cohen-Macaulay ring $R$ is \emph{level} if $\omega_R$ is generated in a single degree and \emph{Gorenstein} if $\omega_R$ is cyclic.

When $R$ is Cohen-Macaulay, the minimal number of generators of $\omega_R$ is called the \emph{type} of $R$ and denoted by $\type R$.  By Grothendieck-Serre duality, we also have 
\begin{equation}
\label{canonical:module:via:Ext}
\omega_R \iso \Ext^c_S(R, \omega_S) \iso \Ext^c_S(R, S)(-n),
\end{equation}
where $c = \hgt I$ and $\dim S = n$ so that applying $\Hom_S(-, S)$ to the minimal free resolution of $R$ over $S$ yields the minimal free resolution of $\omega_R$ up to shifts, and we can therefore read the type of $R$ as the rank of last free module in the minimal free resolution of $R$. 

We will be particularly interested in the case of Artinian rings for examples, so we elaborate on how the above definitions translate to that case.  Recall that the \emph{socle} of $R$ is the ideal $(0: R_+)$.  When $R$ is Artinian, it has finitely many nonzero graded components, and the degree $r$ of the last nonzero component is called the \emph{socle degree} of $R$ as $R_r \subseteq (0:R_+)$.  In this case, we have $\omega_R = {}^*\Hom_\kk(R, \kk)$ so that 
\[ \reg R = a(R) = \max\{j \mid R_j \neq 0 \} \] 
is the socle degree of $R$.  We will therefore use the terminology of regularity, $a$-invariant, and socle degree interchangeably in this case.  Furthermore, $R$ is level if and only if $(0: R_+) = R_r$ and Gorenstein if and only if $(0:R_+)$ is one-dimensional as a $\kk$-vector space.  

Under appropriate conditions, all of these invariants can also be read off from the so-called \emph{$h$-polynomial} $h_R(t)$ of $R$, which is the unique integer polynomial such that the Hilbert series of $R$ can be expressed as a rational function 
\[H_R(t) = h_R(t)/(1-t)^d. \] The \emph{$h$-vector} of $R$ is just the vector \[h(R) = (h_0, h_1, \dots, h_r)\] of coefficients of the $h$-polynomial, where $h_R(t) = \sum_{i=0}^r h_it^i$.  If $R$ is Cohen-Macaulay, then, after extending to an infinite ground field if necessary, we can reduce to the Artinian case by killing a maximal regular sequence of linear forms.  Since this does not affect the Betti numbers or Hilbert series of $R$, we see that the length $r$ of the $h$-vector is none other than regularity of $R$, $h_1 = \codim R$, and if $R$ is level, then $h_r = \type R$.

We close this section with an observation which is relevant to Question \ref{when:are:quadratic:Gorenstein:rings:Koszul}.  

\begin{prop}[{\cite[3.1]{minimal:homogeneous:linkage}}]
Suppose that $R = S/I$ is a quadratic Cohen-Macaulay ring. Then $\reg R \leq \pd_S R$, and equality holds if and only if $R$ is a complete intersection.
\end{prop}

\section{Non-Koszul quadratic idealizations} \label{main:result}

We now come to the central construction of this paper.  Given a ring $R$ and an $R$-module $M$, the \emph{idealization} of $M$ over $R$ is the $R$-algebra $R \ltimes M$ whose underlying $R$-module is $R \oplus M$ with multiplication defined by
\[
(a, x) \cdot (b, y) = (ab, ay + bx)
\]
for all $a, b \in R$ and $x, y \in M$.  In particular, by identifying $a \in R$ and $x \in M$ with $(a, 0)$ and $(0, x)$ in $R \ltimes M$, we view $R$ as a subring of $R \ltimes M$, and the ideal generated by $M$ in $R \ltimes M$ has square zero.

\begin{rmk} \label{standard:graded:idealization:iff:level}
When $R$ is a standard graded algebra and $M$ is a graded $R$-module, the idealization has a natural $\ZZ$-grading given by \[(R \ltimes M)_j = R_j \oplus M_j\] for each $j$.  With this grading, it is clear that the idealization is standard graded if and only if $M$ is generated in degree one.  Since $\omega = \omega_R(-a-1)$ is always nonzero in degree one for $a = a(R)$, we see that the idealization $\tilde{R} = R \ltimes \omega$ of a Cohen-Macaulay ring $R$ is standard graded if and only if $R$ is level.
\end{rmk}

The usefulness of idealization for our purposes is that it gives a canonical way of producing Gorenstein rings from Cohen-Macaulay rings.  The following well known result, adapted here to the standard graded setting, was discovered independently by at least Foxby, Gulliksen, and Reiten; see \cite[7]{Reiten:Gorenstein:idealization} and the lemma preceding Theorem 3 in \cite{Gulliksen}.

\begin{prop} \label{idealization:is:Gorenstein}
If $R$ is a standard graded level $\kk$-algebra, then $\tilde{R} = R \ltimes \omega$ is a  Gorenstein standard graded ring.
\end{prop}

Properties of the level algebra $R$ often carry over to its idealization.  To guarantee that $\tilde{R}$ is still quadratic, we need to impose a slightly stronger condition on $R$ than merely being level.  We say that a standard graded algebra $R$ is \emph{superlevel} if it is level and $\omega_R$ has a linear presentation over $R$.  That is, there is an exact sequence \[R(a-1)^s \stackrel{\bar{\phi}}{\to} R(a)^t \to \omega_R \to 0. \] In particular, every Gorenstein ring is superlevel since $\omega_R \iso R(a)$ in that case.  If $F_1 \stackrel{\phi}{\to} F_0 \to \omega_R \to 0$ is a minimal presentation for $\omega_R$ over $S$, then $\bar{\phi} = \phi \tensor \Id_R$ gives a presentation for $\omega_R$ over $R$, which is minimal up to summands of $F_1 \tensor_S R$ that map to zero.  Hence, $R$ is superlevel if and only if the entries of the matrix of $\phi$ of degree at least two are all contained in $I$.  However, for examples, it will suffice to find rings such that $\omega_R$ has a linear presentation over $S$. 

\begin{lemma} \label{quadratic:idealization:iff:superlevel}
Let $R = S/I$ be a quadratic level algebra.  Then $\tilde{R} = R \ltimes \omega$ is a quadratic algebra if and only if $R$ is superlevel.
\end{lemma}

\begin{proof}
There is an obvious $R$-algebra isomorphism $\Sym_R(\omega)/(\omega)^2 \iso \tilde{R}$. We also have $\Sym_R(\omega) \iso \Sym_S(\omega)/I\Sym_S(\omega)$, and if $\omega$ is minimally generated by $t$ elements, then $\Sym_S(\omega) \iso S[y_1, \dots, y_t]/\mathcal{L}$, where \[\mathcal{L} = (\sum_{i=1}^t f_iy_t \mid (f_1, \dots, f_t) \in \Syz_1^S(\omega)). \] Assembling all of these facts together, we see that
\begin{equation} \label{idealization:presentation}
\tilde{R} \iso S[y_1,\dots, y_t]/((y_1,\dots, y_t)^2 + \mathcal{L} + (I))
\end{equation}
Moreover, since $\tilde{R}_j = R_j \oplus \omega_j$, this isomorphism is graded if we grade $S[y_1, \dots, y_t]$ by total degree in the variables of $S$ and the $y_i$. That is, $(S[y_1, \dots, y_t])_j = \bigoplus_{i \leq j} \bigoplus_{\abs{\alpha} = i} S_{j-i}y^\alpha$, where $y^\alpha = y_1^{\alpha_1}\cdots y_t^{\alpha_t}$ and $\abs{\alpha} = \sum_i \alpha_i$ for all $\alpha \in \NN^t$.  Since $I$ is generated by quadrics, it follows from the above presentation that $\tilde{R}$ is quadratic if and only if the minimal first syzygies of $\omega$ are generated by the linear syzygies and $IS^t$, which happens if and only if the minimal first syzygies of $\omega$ over $R$ are all linear as $\Syz_1^R(\omega) \iso \Syz_1^S(\omega)/IS^t$.
\end{proof}

In order to show that non-Koszulness can be passed from $R$ to its idealization $\tilde{R} = R \ltimes \omega$, we make use of a technical result of Gulliksen computing the graded Poincar\'e series of $\tilde{R}$ in terms of those of $R$ and $\omega$.  The \emph{graded Poincar\'e series} of a finitely generated graded $R$-module $M$ is the formal power series
\[
\P_R^M(s, t) = \sum_{i, j} \beta_{i, j}^R(M)s^jt^i  \in \ZZ[s, s^{-1}][[t]]
\]
When $M = \kk$, we omit the superscript from the notation and refer to $\P_R(s, t)$ as the graded Poincar\'e series of $R$.  Note that $R$ is Koszul if and only if $\P_R \in \ZZ[[st]]$.

\begin{thm}[{\cite[Thm 2]{Gulliksen}}] \label{Poincaré:series:of:the:idealization}
If $R$ is a standard graded $\kk$-algebra and $M$ is a finitely generated graded $R$-module generated in degree one, then the graded Poincar\'e series of $\tilde{R} = R \ltimes M$ is
\[
\P_{\tilde{R}}(s, t) = \frac{\P_R(s, t)}{1 - t\P_R^M(s, t)} 
\]
\end{thm}

Combining this result with the above observations, we have the following.

\begin{thm} \label{non-Koszul:quadratic:superlevel:algebras:yield:non-Koszul:quadratic:Gorenstein:rings}
If $R$ is is a non-Koszul, quadratic superlevel algebra, then $\tilde{R} = R \ltimes \omega$ is a non-Koszul quadratic Gorenstein ring.  Moreover, we have
\begin{align*} 
\codim \tilde{R} &= \codim R + \type R \\
\reg \tilde{R} &= \reg R + 1
\end{align*}
\end{thm}

\begin{proof}
By Proposition \ref{idealization:is:Gorenstein} and Lemma \ref{quadratic:idealization:iff:superlevel}, it suffices to prove that $\tilde{R}$ is not Koszul.  Write
\[ \P_{\tilde{R}}(s, t) = \sum_i f_i(s)t^i \qquad \P_R^\omega(s, t) = \sum_i g_i(s)t^i \qquad \P_R(s, t) = \sum_i h_i(s)t^i \]
for some $f_i, g_i, h_i \in \ZZ[s]$ with non-negative coefficients.  By the above theorem, we know that \[\P_{\tilde{R}}(s, t)(1-t\P_R^\omega(s, t)) = \P_R(s, t)\] so that 
\[
h_i = f_i - \sum_{j = 1}^i f_{i-j}g_{j-1}
\]
for all $i \geq 1$.  Since all of the polynomials in the above expression have non-negative coefficients, any monomial in the support of $h_i$ must also belong to the support of $f_i$.  Since $R$ is not Koszul, there is an $i \geq 1$ such that $h_i$ has a monomial $s^j$ with $j \neq i$ in its support; hence, so does $f_i$, and $\tilde{R}$ is not Koszul.  The statements about the codimension and regularity of $\tilde{R}$ follow from considering the Hilbert series of $\tilde{R}$ and the fact that $\tilde{R}$ is Cohen-Macaulay.
\end{proof}

\begin{rmk}
The interested reader may wish to consult \cite[2.3]{absolutely:Koszul:algebras} where the above argument was also discovered in the context of retracts of rings, of which idealization is a special case.  Part (1) of that result shows that Gulliksen's proof of Theorem \ref{Poincaré:series:of:the:idealization} carries over with minimal changes to general retracts.  We thank Srikanth Iyengar for bringing this paper to our attention.
\end{rmk}

\section{Examples of non-Koszul superlevel algebras} \label{superlevel:examples}

\subsection{Almost complete intersections}

\begin{prop} \label{superlevel:ACI's:in:codimension:4}
If $R = S/I$ is a quadratic Cohen-Macaulay almost complete intersection with $\reg R = 2$ and $\hgt I = 4$, then $R$ is a non-Koszul superlevel algebra. 
\end{prop}

\begin{proof}
Since the conclusion is preserved under flat base change and killing a regular sequence of linear forms on $R$, we may assume without loss of generality that the ground field $\kk$ is infinite and that $R$ is Artinian, and we can choose a quadratic complete intersection $L \subseteq I$ with $\hgt L = \hgt I = 4$.  Set $J = (L : I)$, the ideal directly linked to $I$ by $L$.  Since $R$ is an almost complete intersection, $J$ is a Gorenstein ideal, and
\begin{align*} 
S/J(a) \iso \omega_{S/J} &\iso \Hom_S(S/J, \omega_{S/L}) = \Hom_S(S/J, S/L(4)) \\
&\iso (0:_{S/L} J/L)(4) = I/L(4)
\end{align*}
for $a = \reg S/J$.  As $I/L \iso S/J(a-4)$ is generated in degree two, it follows $a = 2$.  On the other hand, $\omega_R = J/L(4)$ is generated in degrees at least $-2$ so that $J$ is generated in degrees at least 2. Combining this with the fact that $S/J$ is Gorenstein of regularity 2, it follows that $J$ must be generated by quadrics, and in particular, $\omega_R$ is generated in degree exactly $-2$ so that $R$ is level.

The exact sequence $0 \to J/L \to S/L \to S/J \to 0$ yields an induced exact sequence
\[ \Tor_2^S(S/J, \kk)_j \to \Tor^S_1(J/L, \kk)_j \to \Tor_1^S(S/L, \kk)_j = 0 \]
for $j > 2$.  Since $S/J$ has a Gorenstein linear resolution, we see that $\Tor_2^S(S/J, \kk)_j = 0 = \Tor_1^S(J/L, \kk)_j$ for $j > 3$. Hence, $\Tor_1^S(\omega_R, \kk)_j = 0 \mbox{ for }j > -1$, and $\omega_R$ has a linear presentation since it is generated in degree $-2$.

Finally, $R$ is necessarily non-Koszul since any Cohen-Macaulay Koszul almost complete intersection of codimension 4 must have regularity 3 by \cite[3.3]{Koszul:ACI's}.
\end{proof}

Recall that the graded Betti numbers $\beta_{i,j}^S(R)$ may be compactly summarized in the \emph{Betti table} of $R$, where the entry in column $i$ and row $j$ is $\beta_{i,i+j}^S(R)$; the indexing is designed so that the regularity of $R$ is the index of the bottom-most nonzero row in the Betti table of $R$ (compare \eqref{regularity}). This is illustrated in our next example.

\begin{example} \label{counterexample:to:Conca-Rossi-Valla}
As a concrete example of the above proposition, consider the ring $R = S/I$ defined by the ideal 
\[ I = (x^2, y^2, z^2, w^2, xy+zw) \subseteq \kk[x,y,z,w] = S. \] 
To see that $R$ has socle degree 2, it suffices to note that $R_+^3 = 0$.  Since $I$ contains the squares of the variables, it is enough to observe that $I$ contains all four square-free cubic monomials by multiplying $xy+zw$ by each variable.  The Betti table of $R$ is given by 
\vspace{1 ex}
\begin{equation*}
\begin{tabular}{r|ccccc}
 & 0 & 1 & 2 & 3 & 4  \\
\hline  
0 & 1 & -- & -- & -- & -- 
\\ 
1 & -- & 5 & -- & -- & --
\\ 
2 & -- & -- & 15 & 16 & 5
\end{tabular}
\vspace{1 ex}
\end{equation*} 
 Since $\type R = h_2(R) = \binom{5}{2} -5 = 5$, it follows that $\tilde{R} = R \ltimes \omega$ is a non-Koszul quadratic Gorenstein ring with $\codim \tilde{R} = 4 + 5 = 9$ and $\reg \tilde{R} = 2 + 1 = 3$.

Since $R$ is Cohen-Macaulay, it follows from \eqref{canonical:module:via:Ext} that the $S$-dual of the minimal free resolution of $R$ is the minimal free resolution of $\omega$ up to shifting.  Hence, when $\ch(\kk) = 0$, it follows from the proof of Lemma \ref{quadratic:idealization:iff:superlevel} and an explicit computation of the last differential in the minimal free resolution of $R$ in Macaulay2 \cite{Macaulay2} that the generators for the defining ideal of $\tilde{R}$ presented as a quotient of the polynomial ring $S[t_1, \dots, t_5]$ are
\begin{align*}
(x^2, y^2, z^2, w^2, xy &+ zw) + (t_1, \dots, t_5)^2 \\
&+ (xt_1, yt_1, wt_1-xt_2, zt_1 + yt_2, yt_2 + xt_3, yt_3, zt_2 - wt_3, zt_3, \\
& \qquad xt_4, wt_1 - yt_4, wt_2 + zt_4, zt_4 - xt_5, wt_3 - yt_5, zt_5, wt_4, wt_5)
\end{align*}
Furthermore, the Betti table for $\tilde{R}$ is 
\vspace{1 ex}
\begin{equation*}
\begin{tabular}{r|cccccccccc}
 & 0 & 1 & 2 & 3 & 4 & 5 & 6 & 7  &8&9\\
\hline  
0 & 1 & -- & -- & -- & -- & -- & -- & -- &--&--
\\ 
1 & -- & 36 & 160 &330 & 384 & 260 & 96 & 15 &--&--
\\ 
2 & -- & -- & 15 & 96 & 260 & 384 & 330 &160&36&--
\\ 
3 & -- & -- & -- & -- & -- & -- & -- &-- &--& 1 
\end{tabular}
\vspace{1 ex}
\end{equation*} 
\end{example}

\begin{example}
One can show that any quadratic Cohen-Macaulay almost complete intersection $R = S/I$ with $\reg R = n \geq 2$ and $\codim R = 2n$ is non-Koszul and superlevel.  The essential point is that, after possibly extending to an infinite ground field, we can choose a quadratic complete intersection $L \subseteq I$, and any other quadric $q$ such that $I = (L, q)$ will necessarily be a Lefschetz element of degree 2 on $S/L$ by Hilbert function considerations, and conversely, every almost complete intersection $R$ formed by killing a Lefschetz element of degree 2 on a quadratic complete intersection of codimension $2n$ in this way has $\reg R = n$ and $\codim R = 2n$.  The proof that $R$ is superlevel and not Koszul is then essentially Lemma 4.2 of the recent paper \cite{subadditivity:fails:for:Gorenstein:rings}.  

Generalizing the preceding example, one ring of this type is 
\[ R = \kk[x_1,\dots, x_n, y_1, \dots, y_n]/(x^2_i, y_i^2 \mid 1 \leq i \leq n) + (\sum_{i=1}^n x_iy_i), \]
for all $n \geq 2$ by adapting \cite[A.2]{Koszul:properties:of:the:moment:map} to the commutative case.  In \cite{subadditivity:fails:for:Gorenstein:rings}, McCullough and Seceleanu consider another family of rings of this type to show that subadditivity fails for quadratic Gorenstein rings.  Migliore and Mir\`{o}-Roig have also shown in \cite{resolutions:of:generic:ACI's} that generic quadratic almost complete intersections of even codimension are of this form.
\end{example}

\subsection{Ideals of generic forms}
\label{ideals:of:generic:forms}

In this section, we assume that $\kk$ is a field of characteristic zero.  By a \emph{generic} set of $g$ quadrics, we mean a point in a Zariski-open subset of $S_2^g$.  Five generic quadrics in four variables satisfy the conditions of Proposition \ref{superlevel:ACI's:in:codimension:4}, and generic quadrics in more variables provide a larger class of examples of superlevel algebras.

\begin{thm}[{\cite[7.1]{Fröberg:Löfwal:generic:forms}}]
If $R=S/I$ is an Artinian algebra with $I$ generated by $g$ generic quadrics in $n$ variables, then $R$ is Koszul if and only if $g=n$ or $g \geq \frac{n^2+2n}{4}$.
\end{thm}

For an ideal $I$ generated by $g$ generic forms of degree $d$ in $n$
variables, Hochster-Laksov \cite{Hochster:Laksov:generic:forms} prove that $I$ has maximal growth
in degree $d+1$, that is
\[
\dim_\kk I_{d+1} = \min \left\{ gn, {n+d \choose d+1} \right\}
\] 
Consequently, we see that a ring $R = S/I$ defined by $g$ generic quadrics in $n$ variables is non-Koszul and has socle degree 2 if and only if 
\begin{equation} \label{superlevel:generic:forms:condition}
\frac{n^2+3n+2}{6} \leq g < \frac{n^2+2n}{4}
\end{equation} 
The $h$-vector of such an algebra is simply $h(R) = (1,n, \binom{n+1}{2} - g)$.

\begin{figure}[h]
\begin{tabular}{|c|l|}
\hline 
$n$ & $g$ \\ 
\hline 
$4$ & $5$ \\
\hline 
$5$ & $7,8$ \\
\hline 
$6$ & $10,11$ \\
\hline 
$7$ & $12,13,14,15$ \\
\hline 
$8$ & $15,16,17,18,19$ \\
\hline
\end{tabular}
 \caption{Numbers of generic quadrics yielding non-Koszul algebras of socle degree 2 for small $n$}
\end{figure}

\begin{thm} \label{generic:quadrics:with:socle:degree:2}
Let $I \subseteq S = \kk[x_1, \dots, x_n]$ be an ideal generated by $g$ generic quadrics, where $n \geq 4$ and $g$ satisfies the inequalities in \eqref{superlevel:generic:forms:condition}.  Then $R = S/I$ is non-Koszul and superlevel.  Hence, $\tilde{R} = R \ltimes \omega$ is a non-Koszul quadratic Gorenstein ring with $h$-vector
\begin{equation}\label{idealization:of:generic:forms:h-vector}
h(\tilde{R}) = (1, \tfrac{n^2+3n}{2}-g, \tfrac{n^2+3n}{2}-g, 1) 
\end{equation}
\end{thm}

\begin{proof}
Since $I$ is Artinian and has socle degree two, it suffices to show that $\beta_{n-1,n}^S(R) = \beta_{n,n+1}^S(R) = 0$. By upper semicontinuity of the Betti numbers (see \cite[3.13]{upper:semicontinuity}), it further suffices to prove that the corresponding Betti numbers vanish for some initial ideal of $I$. 

Let $J$ denote the initial ideal of $I$ in the degree reverse lexicographic order.  As long as $J$ does not contain the monomials
\begin{equation} \label{avoided:monomials} 
x_1x_{n-1},\ldots, x_{n-2}x_{n-1}, x_{n-1}^2, x_1x_n, \dots, x_1x_{n-1}, x_n^2
\end{equation}
we will have $J_2 \subseteq \kk[x_1,\dots, x_{n-2}]$ so that the projective dimension of $S/(J_2)$ is at most
$n-2$.  In that case, the exact sequence $0 \to J/(J_2) \to S/(J_2) \to S/J \to 0$ then induces exact sequences 
\[ \Tor_i^S(J/(J_2), \kk)_j \to \Tor_i^S(S/(J_2), \kk)_j \to \Tor_i^S(S/J, \kk)_j \to \Tor_{i-1}^S(J/(J_2), \kk)_j \] 
for all $i, j$.   Since $J/(J_2)$ is generated in degree at least 3, we have $\beta_{i,j}^S(J/(J_2)) = 0$ for all $i$ and all $j \leq  i+2$.  In particular, combining this fact with the preceding observations yields $\beta_{i,i+1}^S(S/J) = \beta^S_{i,i+1}(S/(J_2)) = 0$ for $i = n-1, n$ as  wanted.  

Note that the monomials \eqref{avoided:monomials} are the smallest $2n-1$ quadratic monomials in the degree reverse lex order.  Since $I$ is generated by generic forms, we may assume that the determinant of the matrix of coefficients of the $g$ largest monomials for all the generators of $I$ is nonzero, which is a Zariski-open condition on $S_2^g$.  Therefore, after taking suitable $\kk$-linear combinations of generators of $I$, we see that $J$ contains the $g$ largest monomials in the degree reverse lex order, and these monomials must span $J_2$ as $I$ and $J$ have the same Hilbert function.  The $g$ largest quadratic monomials are disjoint from the $2n-1$ smallest so long as $\binom{n+1}{2} - g \geq 2n-1$.  This holds for all $n \geq 7$ by the estimates 
\[ \binom{n+1}{2} - g \geq \binom{n+1}{2} - \floor{\frac{n^2+2n}{4}} = \ceiling{\frac{n^2}{4}} \geq 2n -1 \] 
and by an explicit check when $n = 6$ and $g = 10$.  

In the remaining cases, we cannot use the above argument.  However, the $n = 4$ case follows from Proposition \ref{superlevel:ACI's:in:codimension:4}.  Additionally, when $n = 5$ and $g = 7$, we see that $(1-t)^5h_R(t) = (1-t)^5(1 + 5t +8t^2)$ has no cubic term, which implies that $\beta_{2,3}^S(R) = 0$.  

For the cases $(n, g) = (5, 8), (6, 11)$, we claim that $\beta_{3,4}^S(R) = 0$.  Indeed, we may assume as above that the lead terms of the quadrics generating $I$ are the $g$ largest monomials in degree reverse lex order  $x^{\alpha_1}, \dots, x^{\alpha_g}$.  If $\mathcal{B}$ denotes the set of exponent vectors of the remaining degree two monomials, then we may assume each quadric has the form $x^{\alpha_i} + \sum_{\beta \in \mathcal{B}} c_{i,\beta}x^\beta$ for some $c_{i,\beta} \in \kk$.  

By Schreyer's algorithm \cite{Schreyer's:algorithm:for:free:resolutions}, we can construct a free resolution $F_\bullet$ of $R$ from a Gr\"obner basis including these quadrics.  Now, write $S(-4)^b$ for the number of copies of $S(-4)$ in $F_3$ (which does not depend on the particular coefficients of the quadrics), and consider the portion of the differential $\dd_3: S(-4)^b \to F_2$.  Since we obtain the minimal free resolution of $R$ by pruning $F_\bullet$, we have $\beta_{3,4}^S(R) = 0$ if and only if this submatrix of scalars and linear forms splits, which occurs exactly when the scalar part of $\dd_3$ has rank $b$.  Furthermore, the entries of the scalar part of this submatrix are polynomials in the $c_{i,\beta}$ so that this determines a Zariski-open condition on $S_2^g$ for the vanishing of $\beta_{3,4}^S(R)$.  Therefore, $\beta_{3,4}^S(R) = 0$ for generic sets of quadrics if we can show that there is at least one example with this property, and this is easily checked by a direct computation picking $g$ random quadrics in Macaulay2.
\end{proof}

\begin{example}[{\cite{Roos:non-Koszul:algebras}}] \label{Roos:example}
Not all non-Koszul algebras of socle degree 2 come from this construction.  Roos shows that for $S =  \kk[x, y, z, w, u, v]$ with $\ch(\kk) = 0$ and
\[
I = (x^2,y^2,z^2, u^2,v^2, w^2, xy, yz, uv, vw, xz + 3zw- uw, zw + xu + uw) 
\]
the ring $R = S/I$ is not Koszul.   In this case, the Betti table of $R$ is given by 
\vspace{1 ex}
\begin{equation*}
\begin{tabular}{r|ccccccc}
 & 0 & 1 & 2 & 3 & 4 & 5 & 6 \\
\hline  
0 & 1 & -- & -- & -- & -- & -- & --
\\ 
1 & -- & 12 & 16 & 2 & -- & -- & --
\\ 
2 & -- & -- & 32 & 96 & 100 & 48 & 9
\end{tabular}
\vspace{1 ex}
\end{equation*} 
This ring has $h$-vector $(1,6,9)$, which cannot be realized by generic forms in 6 variables.  Applying the idealization construction to Roos' example, we obtain a non-Koszul quadratic Gorenstein ring $\tilde{R}$ with Betti table
\vspace{1 ex}
\begin{equation*}
\scalebox{0.75}{
\begin{tabular}{r|cccccccccccccccc}
 & 0 & 1 & 2 & 3 & 4 & 5 & 6 & 7 & 8 & 9 & 10 & 11 & 12 & 13 & 14 & 15 \\
\hline  
0 & 1 & -- & -- & -- & -- & -- & -- & -- & -- & -- & -- & -- & -- & -- & -- & --
\\ 
1 & -- & 105 & 896 & 3932 & 11136 &  22154 & 32224 & 34895 & 28224 & 16877 & 7264 & 2134 & 384 & 32 &--&--
\\ 
2 & -- & -- & 32 & 384 &  2134 & 7264 & 16877 & 28224 & 34895 & 32224 & 22154 & 11136 & 3932& 896 & 105 &--
\\ 
3 & -- & -- & -- & -- & -- & -- & -- & -- & -- & -- & -- & -- & -- & -- & -- & 1 
\end{tabular}
}
\vspace{1 ex}
\end{equation*} 
\end{example}

Because Roos' example is superlevel, \eqref{idealization:of:generic:forms:h-vector} does not characterize all $h$-vectors of non-Koszul quadratic Gorenstein rings of regularity three.  This raises the question of which $h$-vectors are possible for such rings.

\begin{thm} \label{regularity:3:negative:answers}
Over a field of characteristic zero, there exist non-Koszul quadratic Gorenstein rings with $h$-vector $(1, c, c, 1)$ for all $c \geq 9$.
\end{thm}

\begin{proof}
For each $n$, the value $c(n,g) =\frac{n^2+3n}{2}-g$ appearing in \eqref{idealization:of:generic:forms:h-vector} is decreasing in $g$ and takes every integer value in the range $c(n, g_{\max}(n)) \leq c \leq c(n, g_{\min}(n))$, where
\[ g_{\max}(n) = \ceiling{\frac{n^2+2n}{4}} - 1 \qquad g_{\min}(n) = \ceiling{\frac{n^2+3n+2}{6}} \]
Hence, there will be no gaps in the codimensions attained by $c(n, g)$ so long as \[c(n, g_{\min}(n)) \geq c(n+1, g_{\max}(n+1)) - 1. \] We claim that this holds for all $n \geq 8$.  This follows from the fact that
\begin{align*} 
c(n, g_{\min}(n)) - c(n+1, g_{\max}(n+1)) &>
c(n, \tfrac{n^2+3n+2}{6}+1) - c(n+1, \tfrac{(n+1)^2+2(n+1)}{4} - 1) \\
&= \frac{n^2-6n-43}{12} \geq -2
 \end{align*}
 when $n \geq 9$ and by an explicit check when $n = 8$.  Thus, the construction of Theorem \ref{generic:quadrics:with:socle:degree:2} yields non-Koszul quadratic Gorenstein rings with $h$-vector $(1, c, c, 1)$ for all $c \geq 25$, and for all $c \geq 9$ with the
exceptions of $c \in \{10,11,14,15,18,19,24\}$. Example \ref{Roos:example} takes care of the $c = 15$ case, and slight modifications yield superlevel non-Koszul algebras of socle degree two in the remaining cases. These cases arise from ideals of the form $I = J + L$, where $J$ is given by Figure \ref{exceptional:cases} below and $L$ is generated by the squares of the variables appearing in the generators of $J$.  The quotients $R = S/I$ corresponding to these examples are easily checked in Macaulay2 to be superlevel with idealization having $h$-vector $(1, c, c, 1)$ for the appropriate $c$ in the list above.
\end{proof}

\begin{figure}[h]
\scalebox{0.95}{
\begin{tabular}{|c|c|}
\hline 
$c$ & $J$ \\ 
\hline 
$10$ & 
\begin{minipage}{0.75\linewidth}
\vspace{1 ex}
\begin{center} 
\hspace{-3 em}
$(x_1x_5, x_1x_2, x_4x_5, x _3x_5+x_1x_4+x_4x_5, x_2x_4+x_3x_5)$
\end{center}
\vspace{1 ex}
\end{minipage} 
\\
\hline 
$11$ &
\begin{minipage}{0.9\textwidth}
\vspace{1 ex}
\begin{center}
\hspace{-3 em}
 $(x_1x_5, x_1x_2, x _3x_5+x_1x_4+x_4x_5, x_2x_4+x_3x_5)$
\end{center}
\vspace{1 ex}
\end{minipage}
\\
\hline 
$14$ & 
\begin{minipage}{0.9\textwidth}
\vspace{1 ex}
\begin{center}
\hspace{-3 em}
$\begin{array}{cc}
(x_1x_2, x_2x_3, x_4x_5, x_5x_6, x_1x_3+3x_3x_6-x_4x_6,  
\\ 
\hspace{9 em}  x_3x_6+x_1x_4+x_4x_6, x_2x_4+x_3x_5)
\end{array}$
\end{center}
\vspace{1 ex} 
\end{minipage} 
\\
\hline 
$18$ & 
\begin{minipage}{0.9\textwidth}
\vspace{1 ex}
\begin{center}
\hspace{-5 em}
$\begin{array}{cc}
(x_1x_2+x_2x_3, x_4x_5, x_5x_6, x_6x_7, x_1x_3+3x_3x_6-x_4x_6,  
\\ 
\hspace{5 em} x_3x_6+x_1x_4+x_4x_6, x_2x_4+x_3x_5,  x_2x_5+x_2x_7,x_1x_7,x_3x_7)
\end{array}$
\end{center}
\vspace{1 ex} 
\end{minipage} 
\\
\hline 
$19$ & 
\begin{minipage}{0.8\textwidth}
\vspace{1 ex}
\begin{center}
\hspace{-3 em}
$\begin{array}{cc}
(x_4x_5, x_5x_6, x_6x_7, x_2x_4+x_3x_5, x_2x_5+x_2x_7, x_1x_7, x_3x_7, 
\\ 
x_7x_8, x_1x_8, x_3x_8, x_4x_8, x_6x_8, x_3x_5, x_2x_7,  
\\
x_1x_2+x_2x_3, x_1x_3+3x_3x_6-x_4x_6,x_3x_6+x_1x_4+x_4x_6) 
\end{array}$
\end{center}
\vspace{1 ex} 
\end{minipage} 
\\
\hline
$24$ & 
\begin{minipage}{0.8\textwidth}
\vspace{1 ex}
\begin{center}
\hspace{-3 em}
$\begin{array}{cc}
(x_1x_2+x_2x_3, x_4x_5, x_5x_6, x_6x_7, x_4x_9, x_5x_7, x_7x_9, x_1x_7, 
\\ 
x_3x_7, x_7x_8, x_1x_8, x_2x_8, x_3x_8, x_5x_8, x_6x_8, 
\\
x_1x_3+3x_3x_6-x_4x_6, x_3x_6+x_1x_4+x_4x_6, x_2x_4+x_3x_5, 
\\
x_2x_5+x_2x_7, x_2x_9+x_1x_9, x_3x_9+x_6x_9)
\end{array}$
\end{center}
\vspace{1 ex} 
\end{minipage} 
\\
\hline
\end{tabular}
}
\label{exceptional:cases}
\caption{Exceptional examples yielding non-Koszul quadratic Gorenstein rings of socle degree 3}
\end{figure}

\begin{rmk}
We have assumed that we are working over a field of characteristic zero in this section in order to simplify the statements of our results.  However, the proof of Theorem \ref{generic:quadrics:with:socle:degree:2} shows that we obtain non-Koszul quadratic Gorenstein rings of regularity 3 and almost every codimension greater than or equal to 9 in all characteristics.  We only need to specify a particular characteristic for the exceptional cases that require a direct computation in Macaulay2.
\end{rmk}

\subsection{More examples via tensor products}

When $R$ is a non-Koszul quadratic Gorenstein ring, the idealization $\tilde{R}$ will again be non-Koszul, quadratic, and Gorenstein with codimension and regularity increased by one.  In this case, \eqref{idealization:presentation} shows that the idealization is just the tensor product $R \tensor_\kk \kk[y]/(y^2)$.  In particular, the results of the previous section show that Question \ref{when:are:quadratic:Gorenstein:rings:Koszul} has a negative answer for all $r \geq 3$ and $c \geq r + 6$.  We can produce more examples by tensoring with other Gorenstein Koszul algebras. 
\begin{prop}
Let $R = S/I$ be a quadratic ring and $B$ be a superlevel Koszul algebra.  Then $R' = R \tensor_\kk B$ is Koszul (resp. level, superlevel) if and only if $R$ is.  Moreover, we have
\begin{align*} 
\codim R' &= \codim R + \codim B \\
\type R' &= (\type R)(\type B) \\
\reg R' &= \reg R + \reg B
\end{align*}
\end{prop}

\begin{proof}
Since tensoring over $\kk$ is exact, tensoring the minimal free resolution of $\kk$ over $B$ with $R$ yields the minimal free resolution of $R$ over $R'$.  As $B$ is Koszul, we see that $\reg_{R '}(R) = 0$ so that $R'$ is Koszul if and only if $R$ is by \cite[\S 3.1, 2]{Koszul:algebras:and:regularity}.  Write $B = A/J$ for some standard graded polynomial ring $A$.  The other parts easily follow from the fact that the minimal free resolution of $R'$ over $S \tensor_\kk A$ is the tensor product of the minimal free resolutions of $R$ over $S$ and of $B$ over $A$.  In particular, when $R$ is level, the equalities concerning the codimension, type, and regularity of $R'$ also follow from the fact that the $h$-polynomial of $R'$ is the product of the $h$-polynomials of $R$ and $B$.
\end{proof}

\begin{cor}
Over a field of characteristic zero, there exists a non-Koszul quadratic Gorenstein ring of codimension $c$ and regularity $r$ for every $r \geq 6$ and $c \geq r + 3$.
\end{cor}

\begin{proof}
If $R$ is a non-Koszul quadratic Gorenstein ring and $B$ is a Gorenstein Koszul algebra, then $R' = R \tensor_\kk B$ is again a non-Koszul quadratic Gorenstein ring.  We can therefore produce more examples of such rings by tensoring Matsuda's example $R$, which has $\codim R = 7$ and $\reg R = 4$, with appropriate Gorenstein Koszul algebras and combining this with our results.  Specifically, if we take any quadratic Gorenstein ring $B$ with $\codim B = 3$, then $\reg B = 2$ by the Buchsbaum-Eisenbud structure theorem for such rings so that $B$ is Koszul, and tensoring with Matsuda's example gives a negative answer to Question \ref{when:are:quadratic:Gorenstein:rings:Koszul} for $(c, r) = (10, 6)$.  If we take $B = \kk[X]/I_2(X)$ where $X$ is a $3 \times 3$ matrix of variables, then the Gulliksen-Neg\r{a}rd resolution \cite[2.5, 2.26]{Bruns:Vetter} shows that $\codim B = 4$ and $\reg B = 2$ so that we also obtain a negative answer for $(c, r) = (11, 6)$.  Propagating these negative answers and the negative answers of Theorem \ref{regularity:3:negative:answers} by tensoring with complete intersections completes the proof.
\end{proof}

We summarize the preceding discussion in Figure \ref{table:à:la:Hartshorne} below.  Aside from the seven remaining cases of codimension $c \geq 6$, one might still hope that every quadratic Gorenstein ring $R$ with $\codim R \leq \reg R + 2$ is Koszul, which could explain the affirmative answers in regularity three.

\begin{figure}[h]
\begin{minipage}{0.75\textwidth}
\begin{tikzpicture}[scale = 0.65]
\draw[gray, very thin] (0, 0.1) grid (12.9, 13);
\draw[->] (-1,13) -- (13,13);
\draw[->] (0, 14) -- (0, 0); 
\node[right] (c) at (13, 13) {$c$};
\node[below] (r) at (0, 0) {$r$};

\foreach \c in {0,1,2,3,4,5,6,7,8,9,10,11} {
	\node[above] (\c) at (\c+1, 13) {\tiny \c};
	\node[left] (\c +12) at (0, 12-\c) {\tiny \c};
	};
	
\foreach \c/\r in {0/0,1/1,2/2,3/3,4/4,5/5,6/6,7/7,8/8,9/9,10/10,11/11,3/2,4/2,5/2,6/2,7/2,8/2,9/2,10/2,11/2,4/3,5/3} {
	\draw[thick, fill = black!75] (\c + 1, 12-\r) circle (0.25);
	};
	
\foreach \c/\r in {6/3,7/3,8/3,6/4,8/4,9/4,7/5, 8/6,9/7,10/8,11/9, 9/5,10/5,5/4,6/5,7/6,8/7,9/8,10/9,11/10} {
	\fill[black!75] (\c+1, 11.75 - \r) -- (\c+1, 12.25 - \r) arc (90:270:0.25) -- cycle;
	\draw[thick] (\c + 1, 12 - \r) circle (0.25);
	};
	
\foreach \c/\r in {9/3,10/3,11/3,10/4,11/4,11/5,7/4,8/5,9/6,10/7,11/8,10/6,11/7,11/6} {
	\draw[thick] (\c + 1, 12-\r) circle (0.25);
	};
	
\draw[thick, fill = black!75] (15, 8) circle (0.25);
\node[right] (y) at (15.5, 8) {\small Yes};
\draw[thick] (15, 7) circle (0.25);
\node[right] (n) at (15.5, 7) {\small No};
\fill[black!75] (15, 5.75) -- (15, 6.25) arc (90:270:0.25) -- cycle;
\draw[thick] (15, 6) circle (0.25);
\node[right] (m) at (15.5, 6) {\small Unknown};
\end{tikzpicture}
\end{minipage}
\caption{Is every quadratic Gorenstein ring of codimension $c$ and regularity $r$ Koszul?}
\label{table:à:la:Hartshorne}
\end{figure}
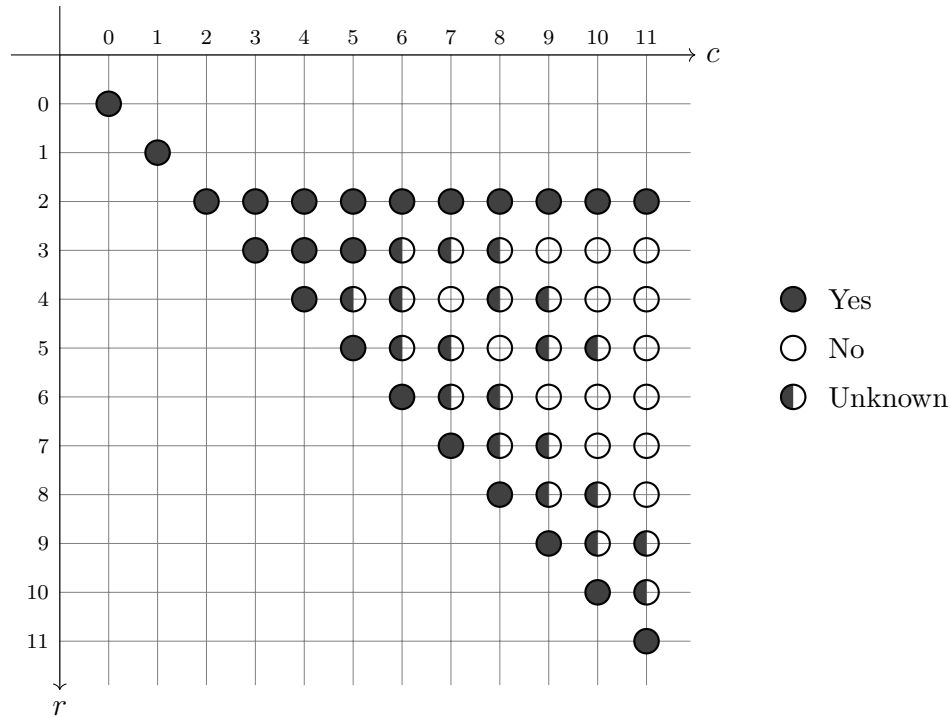

\section{Future Directions}

Matsuda's example cannot be obtained with our methods; there are no superlevel quadratic algebras with the right Hilbert function. It would be interesting to find geometric interpretations of our results. The initial example that inspired the results of this paper was a certain inverse system related to the Artinian reduction of a smooth curve of genus seven and degree eleven in $\PP^5$ defined by five quadrics which is projectively normal but not Koszul \cite{high:rank:linear:syzygies}.  We plan to investigate this further, as well as studying how idealization relates to the parameter space of Gorenstein algebras and work by Iarrobino-Kanev \cite{Iarrobino:Kanev} and Boij \cite{Boij}.  In a followup paper \cite{QGRKP2}, we provide further affirmative(!)\!\! and negative answers to Question \ref{when:are:quadratic:Gorenstein:rings:Koszul} by alternative methods.  More recently, McCullough-Secelanu show in \cite{subadditivity:fails:for:Gorenstein:rings} that idealizing an example of Roos \cite{c=8:r=3}
gives a negative answer to the $(c, r) = (8, 3)$ case. 

Additionally, for Artinian algebras such as those constructed in \S \ref{ideals:of:generic:forms}, much attention has been devoted recently to determining which Gorenstein rings $R$ have the weak Lefschetz property (or WLP); this property asserts the existence of a linear form $\ell \in R_1$ such that the $\kk$-linear multiplication map $R_i \stackrel{\ell}{\to} R_{i+1}$ has maximal rank for all $i$.  The ranks of multiplication maps $R_1 \stackrel{\ell}{\to} R_2$ play an important role in \cite{Gröbner:flags:and:Gorenstein:algebras} and \cite{Caviglia}. We intend to further investigate how WLP may interact with the Koszul property for quadratic Artinian Gorenstein algebras.

\begin{ack}
Macaulay2 computations were essential to our work. We also thank BIRS-CMO, where we learned of Matsuda's result.
\end{ack}

\end{spacing}


\begin{thebibliography}{{\c C}WW95}

\bibitem[ACI10]{free:resolutions:over:Koszul:algebras}
L.~Avramov, A.~Conca, and S.~Iyengar.
\newblock Free resolutions over commutative Koszul algebras.
\newblock {\em Math. Res. Lett.} 17 (2010), no. 2, 197--210. 

\bibitem[Bac86]{high:Veronese:subrings:are:Koszul}
J.~Backelin.
\newblock On the rates of growth of the homologies of Veronese subrings. 
\newblock {\em Algebra, algebraic topology and their interactions (Stockholm, 1983)}, 79--100, 
\newblock Lecture Notes in Math., 1183, Springer, Berlin, 1986.

\bibitem[Bac88]{Backelin}
J.~Backelin.
\newblock Relations between rates of growth of homologies. 
\newblock {\em Research Reports in Mathematics} 25, Department of Mathematics (1988), Stockholm University.

\bibitem[Boi99]{Boij}
M.~Boij.
\newblock Components of the space parametrizing graded Gorenstein Artin algebras with a given Hilbert function.
\newblock {\em Pacific J. Math.} 187 (1999), no. 1, 1--11. 

\bibitem[BHI17]{Koszul:algebras:defined:by:3:quadrics}
A.~Boocher, S.~H.~Hassanzadeh, and S.~Iyengar.
\newblock Koszul algebras defined by three relations. 
\newblock {\em Homological and computational methods in commutative algebra}, 53--68, 
\newblock Springer INdAM Ser., 20, Springer, Cham, 2017. 

\bibitem[BS13]{Brodmann:Sharp}
M.~Brodmann and R.~Sharp.
\newblock {\em Local cohomology.} 
\newblock An algebraic introduction with geometric applications. Second edition. 
\newblock Cambridge Studies in Advanced Mathematics, 136. 
\newblock Cambridge University Press, Cambridge, 2013.

\bibitem[BC02]{upper:semicontinuity}
W.~Bruns and A.~Conca.
\newblock Gr\"obner bases and determinantal ideals. 
\newblock {\em Commutative algebra, singularities and computer algebra (Sinaia, 2002)}, 9--66, 
\newblock NATO Sci. Ser. II Math. Phys. Chem., 115, Kluwer Acad. Publ., Dordrecht, 2003.

\bibitem[BH93]{Bruns:Herzog}
W.~Bruns and J.~Herzog.
\newblock {\em Cohen-Macaulay rings.} 
\newblock Cambridge Studies in Advanced Mathematics, 39. 
\newblock Cambridge University Press, Cambridge, 1993.

\bibitem[BV88]{Bruns:Vetter}
W.~Bruns and U.~Vetter.
\newblock {\em Determinantal rings.} 
\newblock Lecture Notes in Mathematics, 1327. 
\newblock Springer-Verlag, Berlin, 1988.

\bibitem[Cav00]{Caviglia} 
G.~Caviglia. 
\newblock On a class of Gorenstein algebras that are Koszul.
\newblock {\em Laurea in Mathematics (Masters thesis)},
\newblock Universit\`a di Genova, 2000. 

\bibitem[Con14]{Koszul:algebras:and:their:syzygies} 
A.~Conca.
\newblock Koszul algebras and their syzygies. 
\newblock {\em Combinatorial algebraic geometry}, 1--31, 
\newblock Lecture Notes in Math., 2108, Fond. CIME/CIME Found. Subser., Springer, Cham, 2014. 

\bibitem[CDR13]{Koszul:algebras:and:regularity}
A.~Conca, E.~De Negri, and M.E.~Rossi.
\newblock Koszul algebras and regularity. 
\newblock {\em Commutative algebra}, 285--315, Springer, New York, 2013. 

\bibitem[CHI18]{Koszul:properties:of:the:moment:map}
A.~Conca, H.-C.~Herbig, and S.~Iyengar.
\newblock Koszul properties of the moment map of some classical representations.
\newblock {\em Collect. Math.} 69 (2018), no. 3, 337--357. 

\bibitem[CI+15]{absolutely:Koszul:algebras}
A.~Conca, S.~Iyengar, H.~Nguyen, and T.~R\"omer.
\newblock Absolutely Koszul algebras and the Backelin-Roos property.
\newblock {\em Acta Math. Vietnam.} 40 (2015), no. 3, 353--374.

\bibitem[CRV01]{Gröbner:flags:and:Gorenstein:algebras}
A.~Conca, M.~E.~Rossi, G.~Valla.
\newblock Gr\"obner flags and Gorenstein algebras.
\newblock {\em Compositio Math.} 129 (2001), no. 1, 95--121. 

\bibitem[Eis05]{geometry:of:syzygies} 
D. ~Eisenbud, 
\newblock {\em The geometry of syzygies.} 
\newblock A second course in commutative algebra and algebraic geometry. 
\newblock Graduate Texts in Mathematics, 229. 
\newblock Springer-Verlag, New York, 2005.

\bibitem[EK17]{El:Khoury:Kustin}
S.~El Khoury and A.~Kustin.
\newblock Use a Macaulay inverse system to detect an embedded deformation.
\newblock Preprint.

\bibitem[EM+16]{Schreyer's:algorithm:for:free:resolutions}
B.~Er\"{o}cal, O.~Motsak, F.-O.~Schreyer, and A.~Steenpa\ss.
\newblock Refined algorithms to compute syzygies. 
\newblock {\em J. Symbolic Comput.} 74 (2016), 308--327. 

\bibitem[Fr\"o99]{Fröberg:Koszul:algebras:survey}
R.~Fr\"oberg.
\newblock Koszul algebras. 
\newblock {\em Advances in commutative ring theory (Fez, 1997)}, 337--350, 
\newblock Lecture Notes in Pure and Appl. Math., 205, Dekker, New York, 1999. 

\bibitem[FL02]{Fröberg:Löfwal:generic:forms} 
R.~Fr\"oberg and C.~L\"ofwal.
\newblock Koszul homology and Lie algebras with application to generic forms and points. 
\newblock The Roos Festschrift volume, 2. 
\newblock {\em Homology Homotopy Appl.} 4 (2002), no. 2, part 2, 227--258.

\bibitem[M2]{Macaulay2}
D.~Grayson and M.~Stillman.
\newblock Macaulay2, a software system for research in algebraic geometry. 
\newblock Available at \url{http://www.math.uiuc.edu/Macaulay2/}.

\bibitem[Gul72]{Gulliksen} 
T.~Gulliksen.
\newblock Massey operations and the Poincar\'e series of certain local rings. 
\newblock {\em J. Algebra} 22 (1972), 223--232. 

\bibitem[HL87]{Hochster:Laksov:generic:forms}
M.~Hochster and D.~Laksov.
\newblock The linear syzygies of generic forms. 
\newblock {\em Comm. Algebra} 15 (1987), no. 1-2, 227--239. 

\bibitem[HM+07]{minimal:homogeneous:linkage}
C.~Huneke, J.~Migliore, U.~Nagel, and B.~Ulrich.
\newblock Minimal homogeneous liaison and licci ideals. {\em Algebra, geometry and their interactions}, 129--139, 
\newblock Contemp. Math., 448, Amer. Math. Soc., Providence, RI, 2007. 

\bibitem[IK99]{Iarrobino:Kanev}
A.~Iarrobino and V.~Kanev.
\newblock {\em Power sums, Gorenstein algebras, and determinantal loci.} 
\newblock Appendix C by Iarrobino and Steven L. Kleiman. Lecture Notes in Mathematics, 1721. 
\newblock Springer-Verlag, Berlin, 1999.

\bibitem[Kem90]{Kempf}
G.~Kempf.
\newblock Some wonderful rings in algebraic geometry. 
\newblock {\em J. Algebra} 134 (1990), no. 1, 222--224.

\bibitem[Kem92]{points:in:projective:space}
G.~Kempf.
\newblock Syzygies for points in projective space. 
\newblock {\em J. Algebra} 145 (1992), no. 1, 219--223. 

\bibitem[Mas18]{Koszul:ACI's}
M.~Mastroeni. 
\newblock Koszul almost complete intersections. 
\newblock {\em J. Algebra} 501 (2018), 285--302. 

\bibitem[MSS19]{QGRKP2}
M.~Mastroeni, H.~Schenck, M.~Stillman.
\newblock Quadratic Gorenstein rings and the Koszul property II.
\newblock \href{https://arxiv.org/abs/1903.08273}{\tt arXiv:1903.08273 [math.AC]}.

\bibitem[Mat17]{Matsuda} 
K.~Matsuda.
\newblock Non-Koszul quadratic Gorenstein toric rings. 
\newblock {\em Mathematica Scandinavica}, 123 (2018), no. 2, 161--173.

\bibitem[MS20]{subadditivity:fails:for:Gorenstein:rings}
J.~McCullough and A.~Seceleanu.
\newblock Quadratic Gorenstein algebras with many surprising properties. 
\newblock  \href{https://arxiv.org/abs/2004.10237}{arXiv:2004.10237 [math.AC]}.

\bibitem[MM03]{resolutions:of:generic:ACI's}
J.~Migliore and R.~Mir\'o-Roig.
\newblock On the minimal free resolution of $n+1$ general forms.
\newblock {\em Trans. Amer. Math. Soc.} 355 (2003), no. 1, 1--36.
             
\bibitem[Pol95]{Polishchuk} 
A.~Polishchuk.
\newblock On the Koszul property of the homogeneous coordinate ring of a curve. 
\newblock {\em J. Algebra} 178 (1995), no. 1, 122--135. 

\bibitem[Rei72]{Reiten:Gorenstein:idealization}
I.~Reiten.
\newblock The converse to a theorem of Sharp on Gorenstein modules. 
\newblock {\em Proc. Amer. Math. Soc.} 32 (1972), 417--420. 

\bibitem[Roo93]{Roos:non-Koszul:algebras} 
J.-E.~Roos.
\newblock Commutative non-Koszul algebras having a linear resolution of arbitrarily high order. 
\newblock Applications to torsion in loop space homology.
\newblock {\em C. R. Acad. Sci. Paris S\'er. I Math.} 316 (1993), no. 11, 1123--1128. 
 
\bibitem[Roo16]{c=8:r=3}
\newblock J.-E.~Roos. 
\newblock Homological properties of the homology algebra of the Koszul complex of a local ring: examples and questions. 
\newblock {\em J. Algebra} 465 (2016), 399--436.
  
\bibitem[SS12]{high:rank:linear:syzygies} 
H.~Schenck and M.~Stillman.
\newblock High rank linear syzygies on low rank quadrics.
\newblock {\em Amer. J. Math.} 134 (2012), no. 2, 561--579. 

\bibitem[Tat57]{Tate}
 J.~Tate,
\newblock Homology of local and Noetherian rings. 
\newblock {\em Illinois J. Math.} 1, (1957), 14--27.
 
\bibitem[VF93]{Vishik:Finkelberg}
A.~Vishik and M.~Finkelberg.
\newblock The coordinate ring of general curve of genus $g \geq 5$ is Koszul. 
\newblock {\em J. Algebra} 162 (1993), no. 2, 535--539. 

\end{thebibliography}
\end{document}